\documentclass[pdftex, a4paper, 12pt, oneside]{amsart}

\usepackage{geometry}
\usepackage[cp1250]{inputenc}

\usepackage{alltt}
\usepackage{euler}
\usepackage{amsfonts}
\usepackage{amscd}
\usepackage{graphicx}
\usepackage{lpic}
\usepackage{longtable}
\usepackage{url}
\usepackage{import}
\usepackage{amsmath, amscd, amsfonts, amsthm, amssymb, mathrsfs}
\usepackage[sc]{mathpazo}
\usepackage{enumitem, textcomp, setspace}

\usepackage{mathtools}

\usepackage{tikz}
\usetikzlibrary{intersections}
\usepackage{pgfplots}
\usepgfplotslibrary{external}

\allowdisplaybreaks

\numberwithin{equation}{section}

\theoremstyle{plain}
\newtheorem{theorem}{Theorem}[section]
\newtheorem{proposition}[theorem]{Proposition}
\newtheorem{lemma}[theorem]{Lemma}

\theoremstyle{definition}

\newtheorem{remark}[theorem]{Remark}

\usepackage{booktabs}

\usepackage[pagebackref=false]{hyperref}

\usepackage[all]{hypcap}

\definecolor{newblue}{rgb}{0.27, 0.32, 0.86}
\definecolor{newred}{rgb}{0.86, 0.32, 0.27}
\hypersetup{
	pagebackref=true,
	colorlinks=true,       
	linkcolor=newred,          
	citecolor=newblue,        
	filecolor=magenta,      
	urlcolor=newblue           
}

\usepackage{caption}
\usepackage{wrapfig}
\providecommand{\keywords}[1]{\textbf{\textit{Key words and phrases:}} #1}
\providecommand{\subjclass}[1]{\textbf{\textit{2020 Mathematics Subject Classification:}} #1}

\title[\tiny Rigid and shaky hard link diagrams]{Rigid and shaky hard link diagrams}
\author{Micha\l \;Jab\l onowski}

\address{Institute of Mathematics, Faculty of Mathematics, Physics and Informatics, University of Gda\'nsk, 80-308 Gda\'nsk, Poland}
\email{michal.jablonowski@gmail.com}



\date{\today}

\graphicspath{ {fig16/} }

\begin{document}
	
\maketitle

\begin{abstract}

In this study of the Reidemeister moves within the classical knot theory, we focus on hard diagrams of knots and links, categorizing them as either rigid or shaky based on their adaptability to certain moves. We establish that every link possesses a diagram that is a rigid hard diagram and we provide an upper limit for the number of crossings in such diagrams. Furthermore, we investigate rigid hard diagrams for specific knots or links to determine their rigid hard index. In the topic of shaky hard diagrams, we demonstrate the existence of such diagrams for the unknot and unlink, regardless of the number of components, and present examples of shaky hard diagrams.

\end{abstract}


\section{Introduction}

The Reidemeister moves introduced nearly a century ago (in 1926), are studied for their varying aspects and implications. B. Trace \cite{Tra83} noted that two diagrams representing the same knot are connected only by the second and third Reidemeister moves (but not the first) if, and only if, they share the same writhe and winding number (as a curve). O. Oestlund \cite{Ost01} and Hagge's \cite{Hag05} work indicates that for every knot, there exists a pair of diagrams requiring all three move types for transition. A. Coward \cite{Cow09} established that even when all three moves are essential, they can be executed in a specific sequence: initially, only the first moves that increase crossing numbers, followed by the second moves that also increase crossings, then exclusively the third moves, and finally the second moves again, but this time decreasing the number of crossings. M. Lackenby \cite{Lac15} has shown that any trivial knot diagram with $c$ crossings can be transformed into a simple diagram using no more than $(236c)^{11}$ Reidemeister moves.
\par
Hard unknots and unlinks, representing trivial knots and links, require increased complexity before they can be simplified to their minimal diagrams (in terms of number of crossings). In classical knot theory, which examines circles in $3$-dimensional space and their generic diagrams on a $2$-sphere, these concepts have a long history dating back to Goeritz's example \cite{Goe34}. More recent research connects hard unknots to DNA recombination studies (see, e.g. \cite{KauLam06}, \cite{KauLam07}) and tests the accuracy of new upper bounds on the number of Reidemeister moves required to simplify an unknot (see, e.g. \cite{Kau16}). Hard unknots continue to be a central topic in recent research papers (for example \cite{BCLMMSSS23}).


\section{Reidemeister moves and hard diagrams}

From \cite{AleBri26, Rei27} we know that any two marked graph diagrams representing the same type of link are related by a finite sequence of Reidemeister local moves presented in Figure\;\ref{Rmoves} and their mirror moves (and an isotopy of the diagram in $\mathbb{S}^2$).

\begin{figure}[ht]
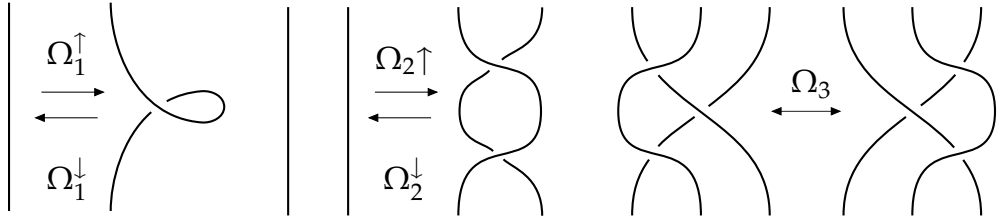

\begin{center}
	\begin{lpic}[]{Rmoves(13cm)}
		\lbl[b]{10,24;$\Omega_1^{\uparrow}$}
		\lbl[t]{10,10;$\Omega_1^{\downarrow}$}
		\lbl[b]{67,24;$\Omega_2{\uparrow}$}
		\lbl[t]{67,10;$\Omega_2^{\downarrow}$}
		\lbl[b]{136,20;$\Omega_3$}
	\end{lpic}
	\caption{The Reidemeister moves.\label{Rmoves}}
\end{center}
\end{figure}

Denote the move $\Omega_1$ as any of $\Omega_1^{\uparrow}$ or $\Omega_1^{\downarrow}$ moves; and denote the move $\Omega_2$ as any of $\Omega_2^{\uparrow}$ or $\Omega_2^{\downarrow}$ moves, also denote the $\Omega_3$-type move as any of $\Omega_3$ or mirror of $\Omega_3$ move. Let us recall the following dependencies between Reidemeister-type moves. The mirror move to the move $\Omega_1$ can be obtained by the moves $\Omega_1, \Omega_2$ and planar isotopy. The mirror move to the move $\Omega_2$ can be obtained by the move $\Omega_2$ and planar isotopy. The mirror move to the move $\Omega_3$ can be obtained by the moves $\Omega_2, \Omega_3$ and planar isotopy.

A \emph{hard diagram} of a link $L$ is a reduced link diagram of $L$ having more than $c(L)$ crossing such that in order to obtain a minimal diagram of $L$ by the Reidemeister moves you have to use at least one move that increases the number of crossings ($\Omega_1^{\uparrow}$ or $\Omega_2^{\uparrow}$). In this paper, the \emph{primeness} of a diagram is considered with respect to the connected sum on the sphere, and the \emph{minimality} of a diagram is considered with respect to the total number of its crossings.

\begin{proposition}
	Any hard diagram is a non-alternating diagram.
\end{proposition}

\begin{proof}
	Let $D$ be a hard diagram of a link $L$. If $L$ is a non-alternating link then by definition $D$ must be non-alternating. If $L$ is an alternating link then any reduced alternating diagram of $L$ is well-known to be minimal (first two Tait conjectures). Therefore, the diagram $D$ cannot be alternating because it is not a minimal diagram for $L$.
\end{proof}

The family of hard diagrams can be further divided into rigid hard and shaky hard families. The rigid hard diagram is a hard diagram such that there is no opportunity to make $\Omega_3$-type move, and it is called shaky otherwise.

\section{Rigid hard diagram}

We recall the known rigid hard diagrams.

\begin{theorem}[\cite{Jab19}]\label{thm1}
	Up to the mirror image, the only minimal rigid hard prime classical unlink diagram with two components is $h8$, and the only minimal rigid hard prime classical unlink diagram with three components is $h12$ (see Figure\;\ref{hard09}). Furthermore, the only minimal rigid hard prime classical unknot diagrams are the four diagrams $h9a, \ldots, h9d$ shown also in Figure\;\ref{hard09}.
\end{theorem}

\begin{figure}[ht]
	\includegraphics[width=0.325\textwidth]{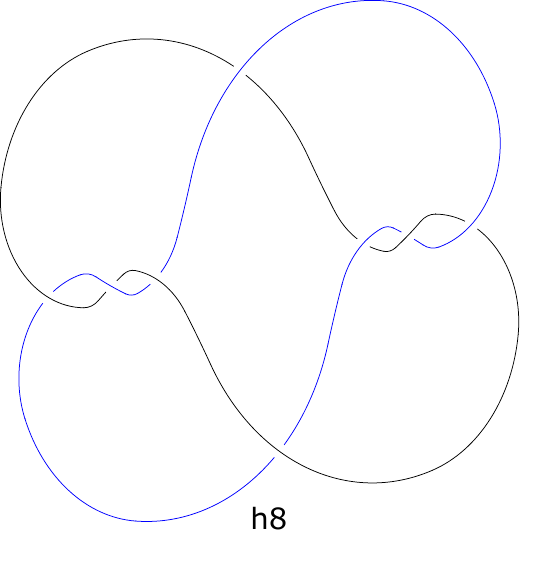}
	\includegraphics[width=0.325\textwidth]{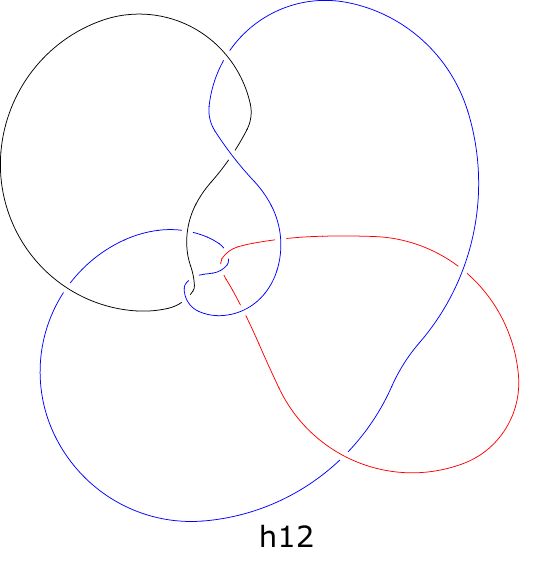}
	\includegraphics[width=0.325\textwidth]{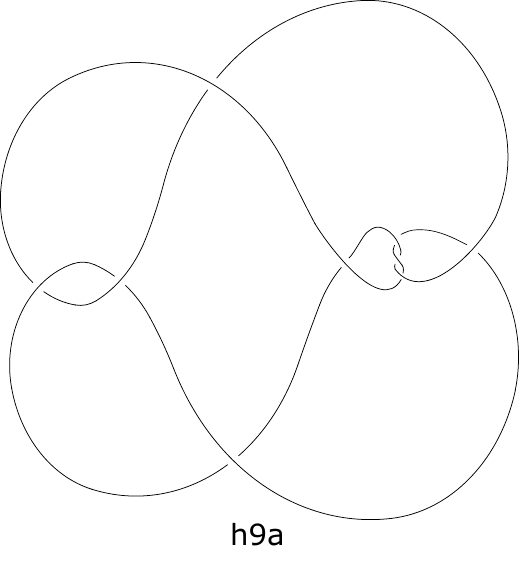}\\
	\includegraphics[width=0.325\textwidth]{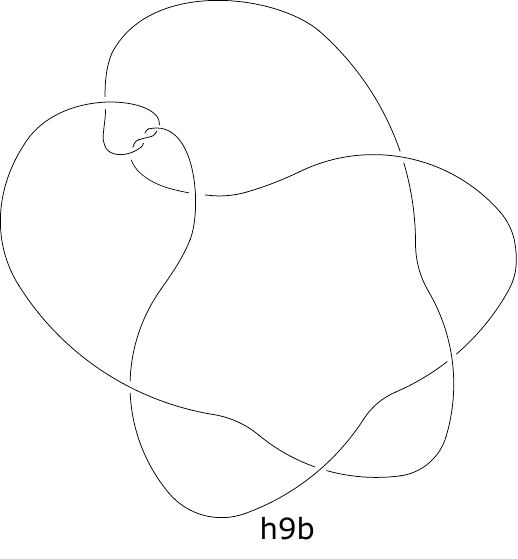}
	\includegraphics[width=0.325\textwidth]{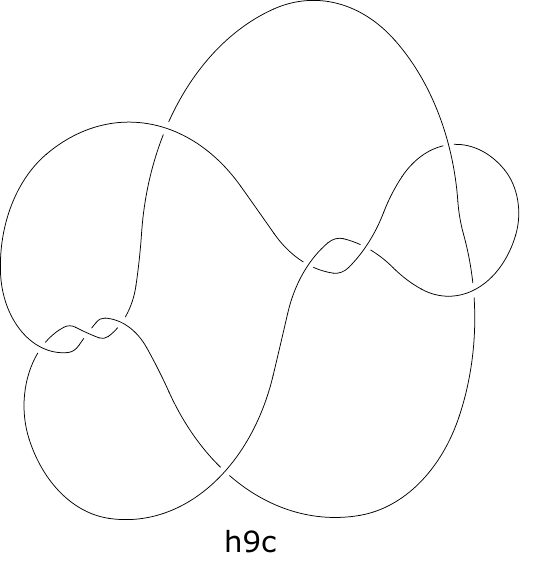}
	\includegraphics[width=0.325\textwidth]{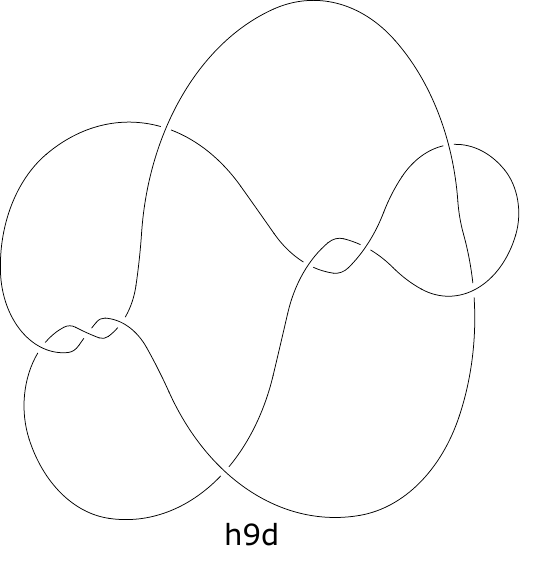}
	\caption{Minimal hard prime unlink diagrams.\label{hard09}}
\end{figure}

From the computational checking of all spherical diagrams up to several crossings, we conclude the following.

\begin{lemma}\label{lem1}
The diagram $7s$ shown in Figure\;\ref{7sumC} is a minimal unknot diagram such that all regions with the opportunity to make $\Omega_2^{\downarrow}$ move or $\Omega_3$-type move share exactly one edge (shown in dashed line). 
\end{lemma}

\begin{remark}
	From the Lem.\;\ref{lem1} we know that we can drop the word "prime" in Thm.\;\ref{thm1} because all the diagrams there have less than $14$ crossing, so the non-prime diagrams with the assumptions of the Theorem are not minimal. 
\end{remark}

\begin{figure}[ht]
	\includegraphics[width=0.5\textwidth]{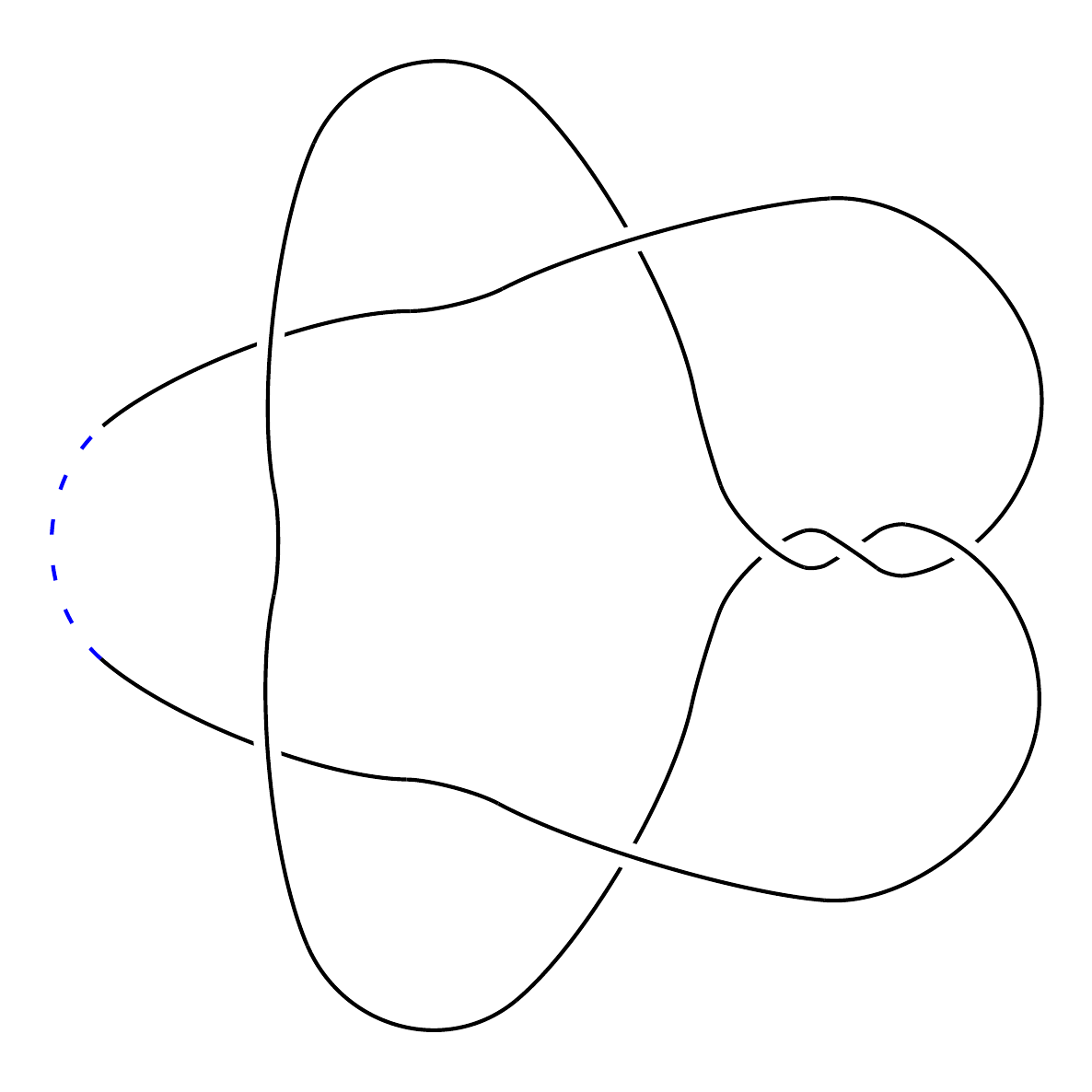}
	\caption{The diagram $7s$.\label{7sumC}}
\end{figure}

In a Figure\;\ref{triangle_regions} we have all types of local neighborhoods of a triangular region in a link diagram we have exactly two cases (in the first row) where a triangular region has the opportunity to make the $\Omega_3$-type move, call it $\Omega_3$\emph{-triangle}.

\begin{theorem}\label{thm2}
	Any link $L$ has a diagram $D$ that is a rigid hard diagram. Moreover, we can choose $D$ such that we have: 
	$$ \#\text{crossings}(D)\leq 7\cdot tri(L) + c(L),$$ for any non-split non-trivial link $L$, where $tri(L)$ is the minimal number of $(\Omega_3\text{-triangles})$ in a minimal diagram of $L$, or is equal to $1$ where there are no such triangles.
\end{theorem}
\begin{proof}
	For a case where the link $L$ is trivial we can take diagrams from Thm.\;\ref{thm1} and their proper amount of connected sums to obtain a rigid hard diagram with the desired number of components. Assume, from now that $L$ is non-trivial. If it is a split link then we can obtain the following procedure for every sub-diagram component obtaining a rigid hard diagram as a disjoint sum of rigid hard diagrams of non-split components. Therefore, we assume that we have a minimal diagram $D'$ for $L$ which is a connected diagram and has at least two crossings. 

\begin{figure}[ht]
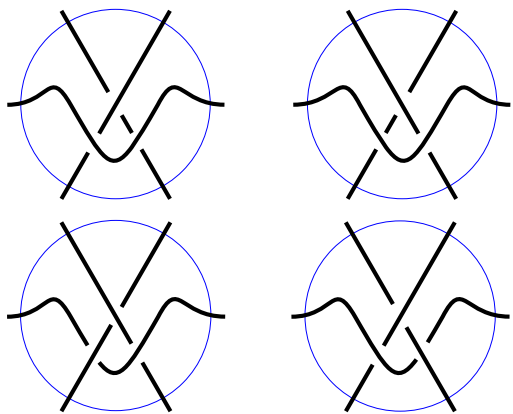

	\begin{center}
		\begin{lpic}[]{triangle_regions(14cm)}

		\end{lpic}
		\caption{Types of triangular regions.\label{triangle_regions}}
	\end{center}
\end{figure}

	Consider $D$ as a connected sum of a minimal diagram $D'$ for $L$ with the diagram $7s$ (along the dashed line) inside every triangular region of $D'$ where there is an opportunity to make $\Omega_3$-type move with the edge used in the sum edge being one of the edges bounding the triangular region. If there is no such region we make only one connected sum with the arbitrary edge of $D'$ and the inequality is obtained. From now on let us assume that there is a least one triangular region with the opportunity to make the $\Omega_3$-type move, call it $\Omega_3$-triangle. Then, we obtain no new region having less than $4$ edges in the boundary, so the diagram of the sum is a rigid hard diagram.
	\par
	The diagram $7s$ has $7$ crossings, therefore in $D$ we have $$7\cdot\#(\Omega_3\text{-triangles})+\# \text{crossings}(D')$$ crossings in $D$.
	 
\end{proof}

We can derive the following general upper bound (without knowledge of the number of $\Omega_3\text{-triangles}$).

\begin{proposition}\label{cor1}
	Any link $L$ has a diagram $D$ that is a rigid hard diagram. Moreover, we can choose $D$ such that we have: 
	$$ \#\text{crossings}(D)\leq 8\cdot c(L),$$ for any non-split non-trivial link $L$.
\end{proposition}

\begin{proof}
	Let $D'$ be a minimal diagram for $L$, having $n=c(L)$ crossings. 
	
	By the Euler characteristic formula for spherical $4$-valent graphs (here connected projections of the links), the number of all regions in $D'$ is $n+2$, and the number of edges in $2n$. Let us consider only triangular regions in the graph $G$ (being a connected shadow for $D'$), We glue together any two such triangles to form a quadrilateral, and then we proceed to any other possible pairs, etc. and we end with only triangles isolated from each other. Call the set of such isolated triangles by $S$, and the set of constructed quadrilaterals by $Q$. We can now apply Theorem\;\ref{thm2} for only $(\Omega_3\text{-triangles})$ in $S$ and for one $(\Omega_3\text{-triangles})$ for each quadrilateral in $Q$ making the connected sum in an edge used to glue our quadrilateral from two triangles, preventing both of them to be $(\Omega_3\text{-triangles})$. At least one (the diagonal) edge in each element of $Q$ is not the edge of any element of $\#S$ and we have a total of $2n$ in a graph in $G$, therefore $\#S\leq\frac{2n-\#Q}{3}$. We also know that each element in $Q$ covers two faces of the starting graph $G$ with $n+2$ faces, so $\#Q\leq\frac{n}{2}+1$.
	We can now count and bind the number of required connected sum operations $CS$ as follows:
	$$CS=\#S+\#Q\leq\frac{2n-\#Q}{3}+\#Q=\frac{2}{3}(n+\#Q)\leq \frac{2}{3}(n+\frac{n}{2}+1)=n+\frac{2}{3}.$$
	The number $CS$ is an integer so we obtain $CS\leq n=c(L)$.
		
\end{proof}

We see that in the proof of the previous proposition, the construction leads us to generally not prime diagram. It is natural then to search for possibly prime rigid hard diagrams for a given knot or link with fewer crossings.
\par
In Figure\;\ref{exe1} it is shown an example of a minimal rigid hard diagram of the trefoil diagram. Notice that it is a prime diagram and it has $9$ crossings while the minimal diagram of the trefoil has $3$ crossings.

\begin{figure}[ht]
	\includegraphics[width=0.6\textwidth]{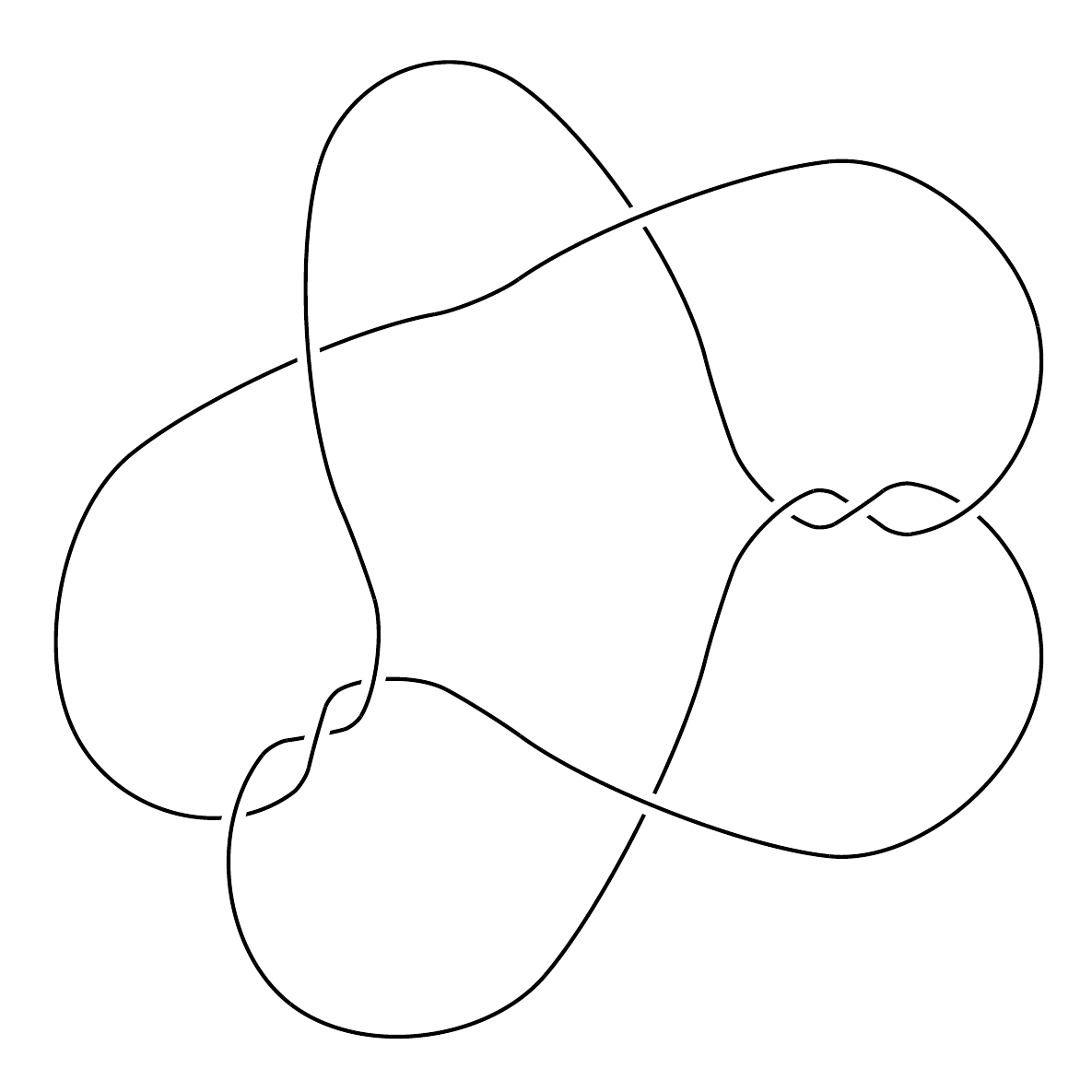}
	\caption{An example of a minimal rigid hard diagram of the trefoil knot.\label{exe1}}
\end{figure}

The \emph{rigid hard index} of a link $L$, denoted by $ind_{rh}(L)$ is the difference between the number of crossings in a minimal (with respect to the number of crossing) rigid hard diagram for $L$ and the crossing number of $L$.

From the previous examples of prime minimal rigid hard diagrams of trivial links. We can conclude directly that: 
$$ind_{rh}(T_1)=9,\; ind_{rh}(T_2)=8,\; ind_{rh}(T_3)=12.$$

From Thm.\;\ref{thm2} we have in particular that for any alternating (or generally for links not having any $\Omega_3$ triangle in a minimal diagram) non-trivial non-split links $L$, we have:

$$ind_{rh}(L)\leqslant 7.$$

We also directly know that for any non-split non-trivial link $L$, we have:

$$ind_{rh}(L)\leqslant 7\cdot c(L).$$

We computationally generate (with checking all diagrams) the following Table\;\ref{table2} of the rigid hard index of prime knots and links. The table consists of prime knots and links with the crossing number less than $8$.
\par
We consider links (and their names) up to the mirror image because the rigid hard index of a given link and the rigid hard index of its mirror image are the same. 

\begin{footnotesize}
	\renewcommand{\arraystretch}{1.25}
	\begin{center}
		
		\begin{longtable}[ht]{r||l}
			\caption{Knots and links and the rigid hard index\label{table2}}\\
			$ind_{rh}(L)$	&  names of prime knots or links $L$ with $c(L)\leq 8$\\
			\hline
			\endfirsthead
			\multicolumn{2}{c}
			{\tablename\ \thetable\ -- \textit{Continued from previous page}} \\
			$ind_{rh}(L)$	 & names of prime knots or links $L$ with $c(L)\leq 8$\\
			\hline
			\endhead
			\hline \multicolumn{2}{r}{\textit{Continued on next page}} \\
			\endfoot
			\hline
			\endlastfoot
			$1$& L5a1, K6a1, K6a2, L6a5, K7a1, K7a2, K7a6, L7a2, L7a3, L7a4, L7a5, L7a6, L7a7, L7n2,  \\
			&
			K8a1, K8a2, K8a3, K8a4, K8a5, K8a6, K8a7, K8a8, K8a9, K8a10, K8a13, K8a16, K8a17,\\&
			 K8n1, K8n2, K8n3,
			 L8a1, L8a2, L8a3, L8a4, L8a5, L8a6, L8a8, L8a9, L8a10, L8a15,\\&
			 L8a17, L8a18, L8a20, L8a21, L8n1, L8n2, L8n4, L8n7 
			\\
			\hline

			$2$& L6a1, L6a3, L6a4, K7a3, K7a4, K7a7, L7a1,  K8a11, K8a12, K8a14, L8a7, L8a11, L8a13,\\& L8a16, L8a19, L8n3, L8n5, L8n6
			\\

			\hline
			$3$&
			K6a3, L6n1, K7a5, K8a15, K8a18
			
			\\
			\hline
			$4$&
			K5a2
			
			\\
			
			\hline
			$5$& 
			L4a1, K5a1
			\\
			\hline
			$6$& 
			K3a1, K4a1
			\\
			\hline
			$7$& 
			L2a1
			\\
		\end{longtable}
	\end{center}
	
\end{footnotesize}

\section{Shaky hard diagram}

From the introduction section, we know that any hard diagram of any link must be a non-alternating diagram. In the case of shaky hard diagrams, it can be deduced directly because the opportunity to make the Reidemeister III move requires that at least one strand (involved in the move) have two crossings that do not alternate.

\begin{theorem}
	There exists a shaky hard diagram of the unknot and unlink of arbitrary many components.
\end{theorem}

\begin{proof}

In Figure\;\ref{exe2} it is shown an example of a prime shaky hard diagram of the unknot (it has $10$ crossing so one can test the triviality by calculating its Jones polynomial). The case for two-component unlink examples are shown in Figure\;\ref{exe3} (notice that those diagrams of unlinks differ by crossing changes at two crossings). 
\par
They are all shaky diagrams because each has a triangular region, marked by the letter T that has the opportunity to make the $\Omega_3$-type move. They are also hard because each diagram has exactly one such triangular region and after performing the $\Omega_3$-type move it can be easily checked that the resulting diagrams have also exactly one triangular region, that has the opportunity to make the $\Omega_3$-type move and no opportunity to make the $\Omega_1^{\downarrow}$ or $\Omega_2^{\downarrow}$ move. Therefore to try to reduce it further without increasing the number of crossings one has the potential move to make on that one triangular region resulting in returning to the original started diagram.
\par
For the shaky hard diagram of the unlink of arbitrary many components, it is sufficient to take a diagram from Figure\;\ref{exe2} and make connected sums of with the sufficient amount of the diagram $h8$ from Figure\;\ref{hard09}, with the connection arc being the far from (i.e. not touching) the region $T$ and its neighbor regions.

\begin{figure}[ht]
	\includegraphics[width=0.5\textwidth]{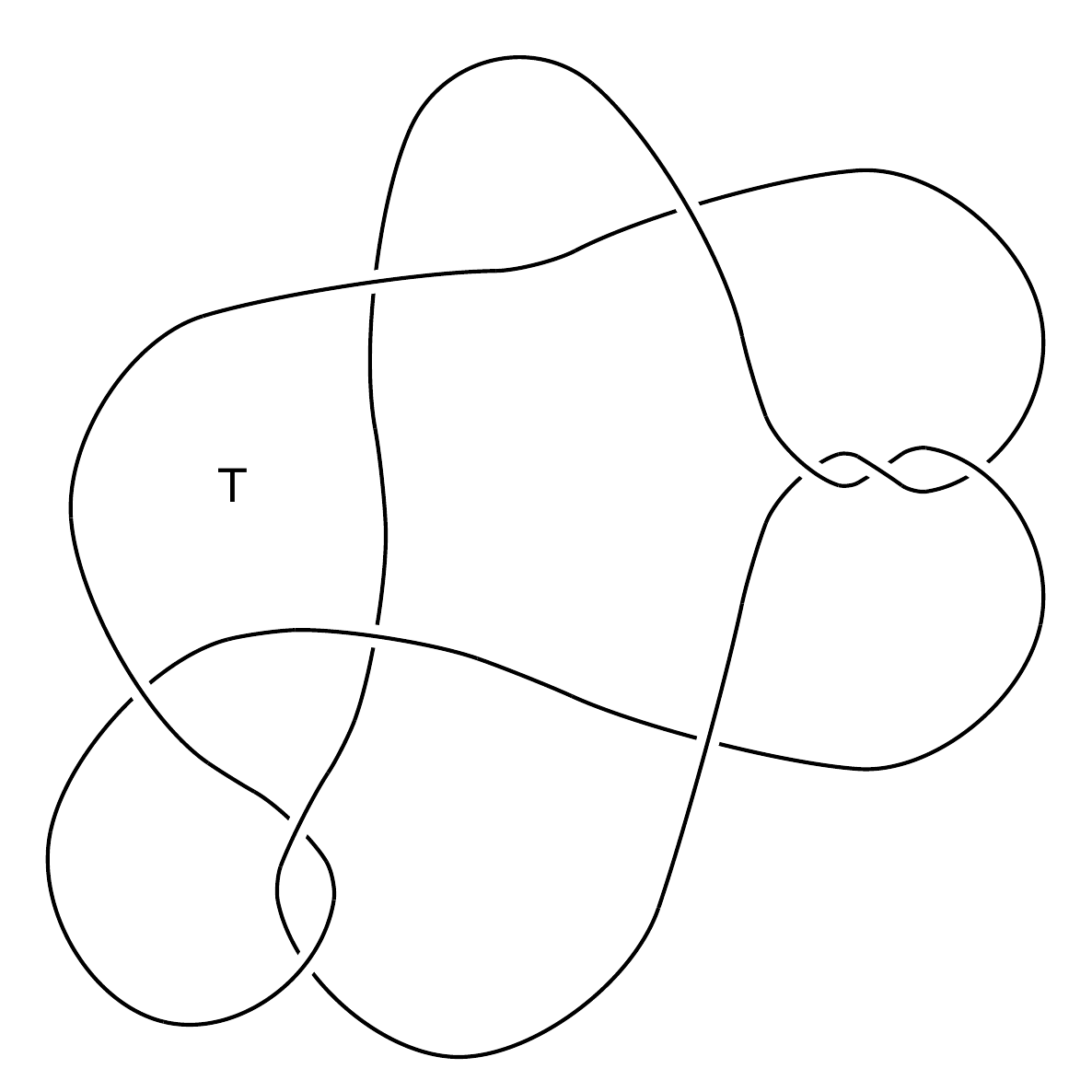}
	\caption{An example of a prime shaky hard diagram of the trivial knot.\label{exe2}}
\end{figure}

\begin{figure}[ht]
	\includegraphics[width=0.44\textwidth]{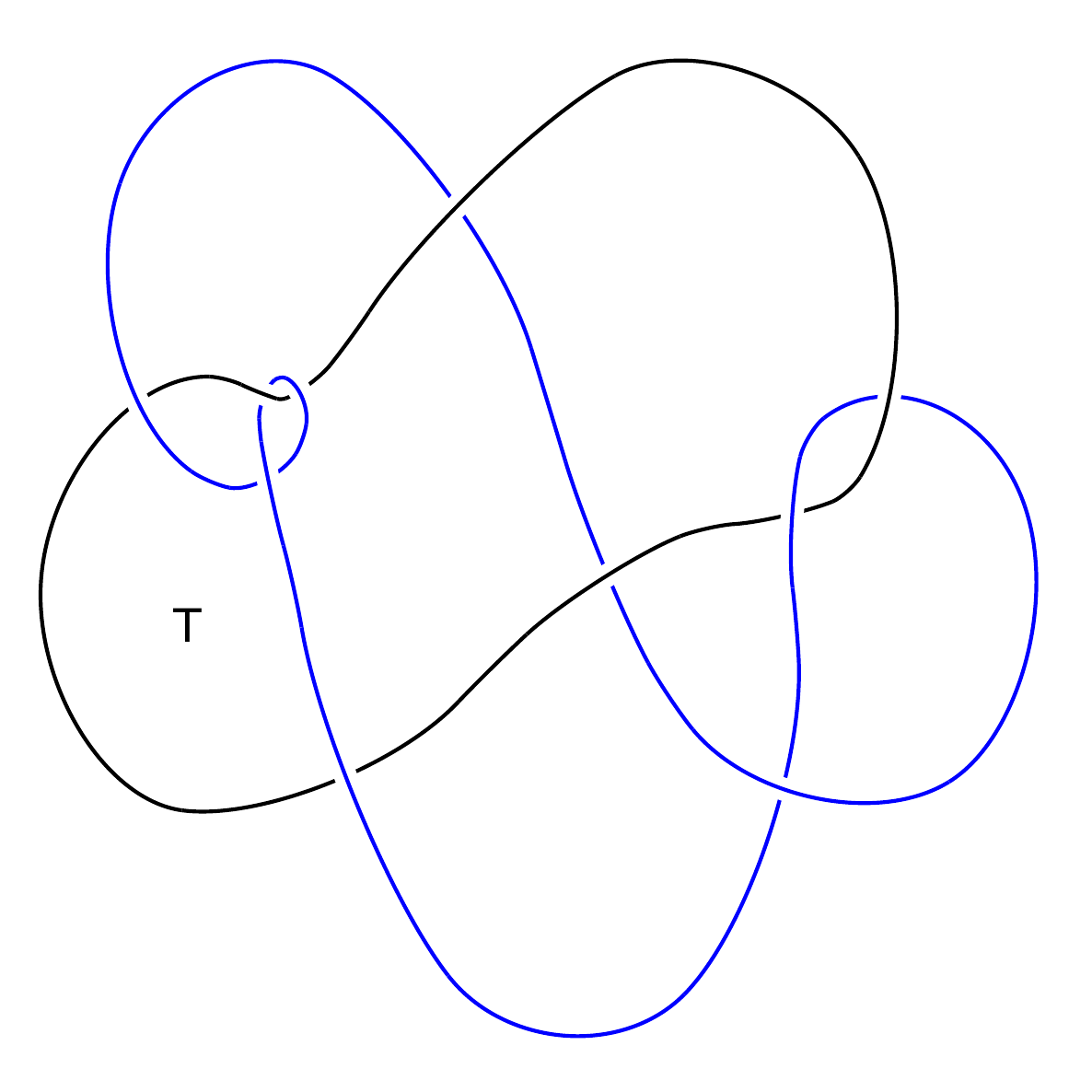}
	\includegraphics[width=0.44\textwidth]{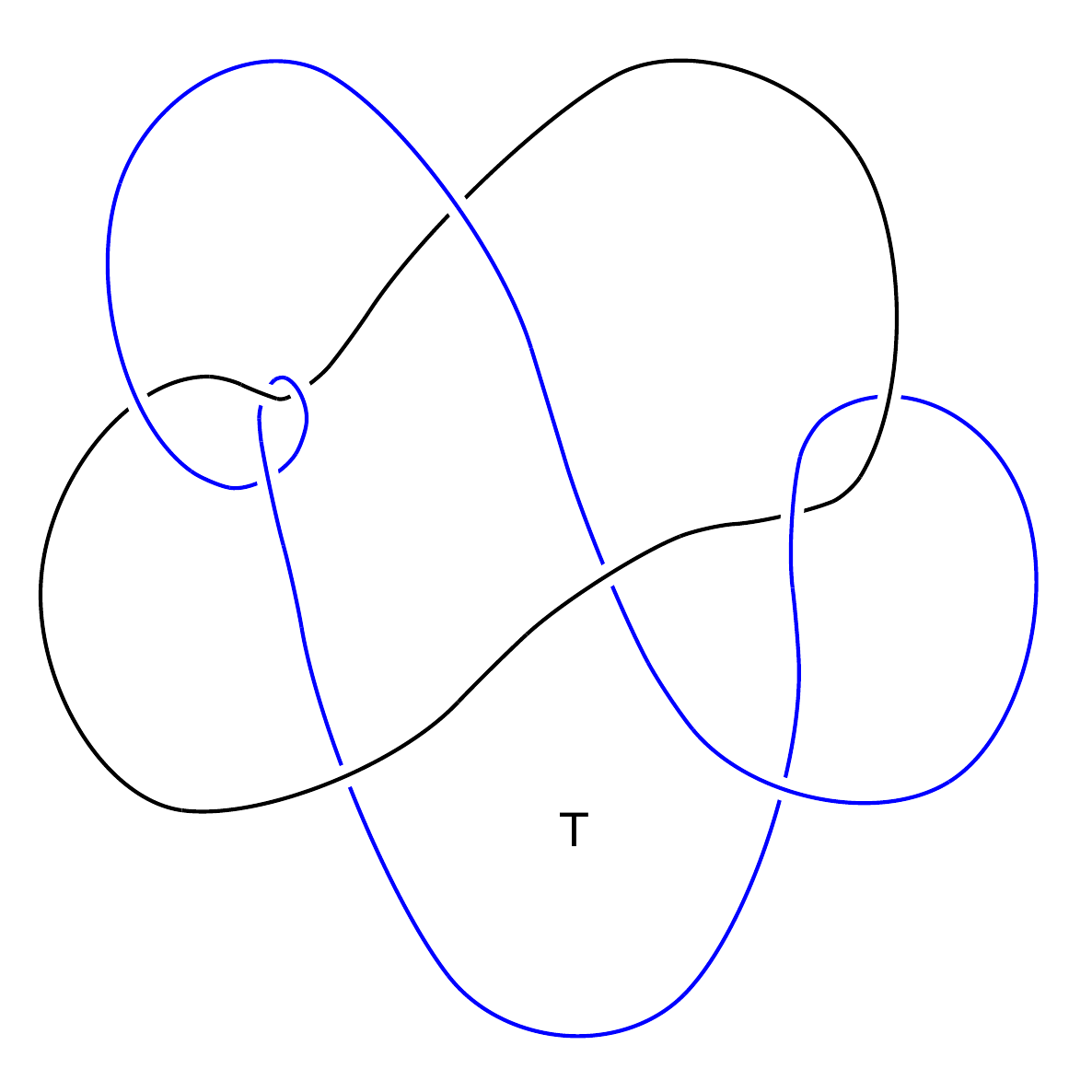}
	\caption{Examples of a prime shaky hard diagram of a trivial link.\label{exe3}}
\end{figure}
	
\end{proof}

\begin{figure}[ht]
	\includegraphics[width=0.47\textwidth]{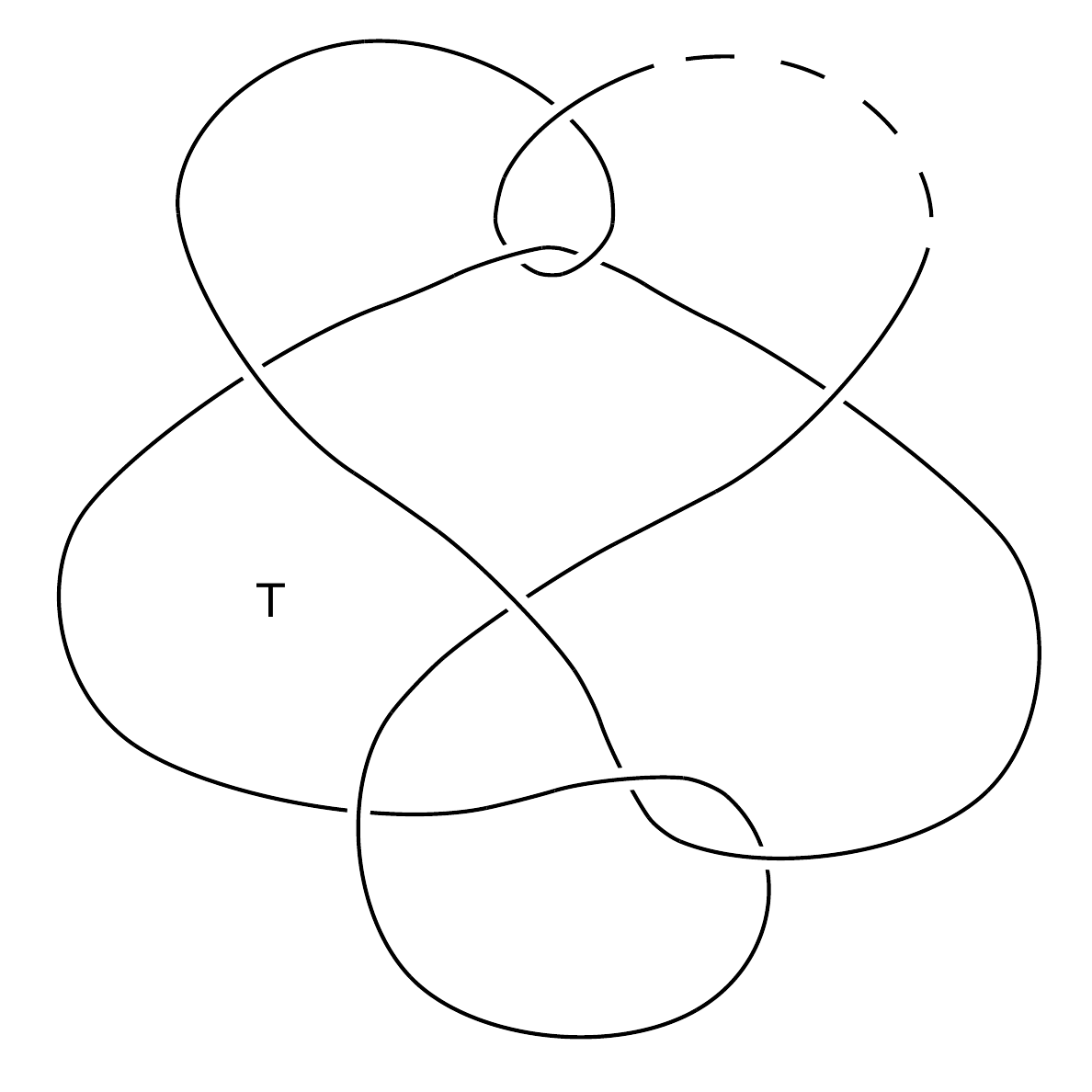}
	\caption{The diagram $9s$.\label{9s}}
\end{figure}

By a similar construction as in Theorem\;\ref{thm2} but instead of one diagram $7s$ we take the connected sum with the diagram $9s$ shown in Figure\;\ref{9s} (along the dashed line), we obtain the following.

\begin{proposition}\label{prop1}
	Any link $L$ has a diagram $D$ that is a shaky hard diagram. Moreover, we can choose $D$ such that we have: 
	$$ \#\text{crossings}(D)\leq 7\cdot tri(L)+2+c(L),$$ for any non-split non-trivial link $L$, where $tri(L)$ is the minimal number of $(\Omega_3\text{-triangles})$ in a minimal diagram of $L$, or is equal to $1$ where there are no such triangles.
\end{proposition}

\pagebreak


\begin{thebibliography}{999}

\bibitem{AleBri26} J.W. Alexander and G. B. Briggs, On Types of Knotted Curves, \emph{Annals of Mathematics, Second Series} 28 (1926-1927), 562--586.

\bibitem{BCLMMSSS23} B.A. Burton, H.C. Chang, M. Löffler, C. Maria, A de Mesmay, S. Schleimer, E. Sedgwick and J. Spreer, Hard Diagrams of the Unknot. \emph{Experimental Mathematics} (2023), 1--19.

\bibitem{Cow09} A. Coward, Ordering the Reidemeister moves of a classical knot \emph{Algebraic and Geometric Topology} 6 (2006), 659--671.

\bibitem{Goe34} L. Goeritz, Bemerkungen zur knotentheorie, \emph{Abh. Math. Sem. Univ. Hamburg} 10 (1934), 201--210.

\bibitem{Hag05} T. Hagge, Every Reidemeister move is needed for each knot type. \emph{Proceedings of the American Mathematical Society} 134 (2005), 295--301.

\bibitem{Jab19} M. Jab\l onowski, Minimal hard surface-unlink and classical unlink diagrams, \emph{J. Knot Theory Ramifications} 28 (2019), 1940002.

\bibitem{Kau16} L. Kauffman, The Unknotting Problem, in \emph{Open Problems in Mathematics}, Edited by J.F. Nash and Jr., M.Th. Rassias, Springer International Publishing Switzerland (2016) 303-345.

\bibitem{KauLam06} L. Kauffman, S. Lambropoulou, Unknots and molecular biology, \emph{Milan Journal of Mathematics} 74 (2006), 227--263.

\bibitem{KauLam07} L. Kauffman, S. Lambropoulou, Unknots and DNA, in \emph{Current Developments in Mathematical Biology: Proceedings of the Conference on Mathematical Biology and Dynamical Systems, the University of Texas at Tyler, 7--9 October 2005}, Vol. 38 World Scientific (2007), 39--68.

\bibitem{Lac15} M. Lackenby, A polynomial upper bound on Reidemeister moves, \emph{Annals of Mathematics} (2015), 491--564.

\bibitem{Ost01} O.P. Östlund, Invariants of knot diagrams and relations among Reidemeister moves, \emph{J. Knot Theory Ramifications}, 10(8) (2001), 1215--1227.

\bibitem{Rei27} K. Reidemeister, Elementare Begründung der Knotentheorie, \emph{Abh. Math. Semin. Univ. Hamburg} 5 (1927), 24--32.

\bibitem{Tra83} B. Trace, On the Reidemeister moves of a classical knot, \emph{Proc. Amer. Math. Soc.} 89 (1983), 722--724.

\end{thebibliography}
\end{document}